\documentclass[english]{amsart}
\usepackage[utf8]{inputenc}
\usepackage{amssymb,amsmath,amsthm,amsrefs,newlfont,enumerate,array,graphicx,tikz,tkz-tab,xcolor,bbold,float,subcaption}
\usepackage{mathrsfs}
\usepackage[colorlinks, citecolor=blue, linkcolor=red, pdfstartview=FitB]{hyperref}
\usepackage[noabbrev]{cleveref} 
\usepackage[shortlabels]{enumitem}
\makeatletter
\newcommand{\customlabel}[2]{\def\@currentlabel{#2}\label{#1}}
\makeatother

\newtheorem{theorem}{Theorem}[section]
\newtheorem*{theorem*}{Theorem}
\newtheorem{lemma}[theorem]{Lemma}
\newtheorem{corollary}[theorem]{Corollary}
\newtheorem{fact}[theorem]{Fact}
\newtheorem{proposition}[theorem]{Proposition}

\theoremstyle{definition}
\newtheorem{definition}[theorem]{Definition}
\newtheorem{remark}[theorem]{Remark}
\newtheorem{example}[theorem]{Example}

\newcommand{\ps}[1]{\left<#1\right>}
\let\Im\relax\DeclareMathOperator{\Im}{Im}
\let\Re\relax\DeclareMathOperator{\Re}{Re}
\newcommand{\bb}{\mathbb}

\DeclareMathOperator{\spa}{span}
\DeclareMathOperator{\w}{wind}

\title[Dynamical properties of Toeplitz operators]{Introduction to the dynamical properties of Toeplitz operators on the
Hardy space of the unit disc.}

\author[Fricain]{Emmanuel Fricain}
 \address{Univ. Lille, CNRS, UMR 8524 - Laboratoire Paul Painlevé, F-59000 Lille, France}
 \email{emmanuel.fricain@univ-lille.fr}

\author[Ostermann]{Ma\"eva Ostermann}
\address{Univ. Lille, CNRS, UMR 8524 - Laboratoire Paul Painlevé, F-59000 Lille, France}
\email{maeva.ostermann@univ-lille.fr}
\begin{document}


\thanks{This work was supported in part by the project COMOP of the French National Research Agency (grant ANR-24-CE40-0892-01) and a variety of sponsors of the ACOTCA conference you could find here: \href{https://indico.math.cnrs.fr/event/13430/}{https://indico.math.cnrs.fr/event/13430/}. The authors acknowledge the support of the CDP C$^2$EMPI, as well as of the French State under the France-2030 program, the University of Lille, the Initiative of Excellence of the University of Lille, and the European Metropolis of Lille for their funding and support of the R-CDP-24-004-C2EMPI project. The second author also acknowledges the support of the CNRS}

\maketitle
\begin{abstract}
These notes are based on a mini-course given at the ACOTCA conference 2025. The goal is to present full proofs of the first two key results regarding hypercyclic Toeplitz operators, in a way that is accessible to beginners. 
\end{abstract}

\section{Introduction}

These notes are based on a short mini-course, consisting of three 50-minute lectures, which the fist author delivered at the ACOTCA (XIX Advanced Course in Operator Theory and Complex Analysis) held in June 2025 in Clermont-Ferrand, France. The topic concerns the hypercyclicity of Toeplitz operators on the Hardy space $H^2$ of the open unit disc $\mathbb D$. 

These notes are intended neither as a comprehensive survey of the literature on this subject nor as a presentation of new results. Rather, their purpose is to provide an introductory account accessible to readers approaching the topic for the first time. To this end, the exposition is aimed at readers (participants) with a background corresponding to a standard master's level course in functional analysis and complex analysis and includes complete proofs of most of the results discussed.

The central question addressed in these notes is the following: given a function $\phi\in L^\infty(\mathbb T)$, where $\mathbb T=\partial \mathbb D$ denotes the unit circle equipped with the normalized Lebesgue measure $m$, can one determine some necessary and/or sufficient conditions on $\phi$ to ensure that the associated Toeplitz operator $T_\phi$ is hypercyclic on $H^2$? 

The structure of this mini-course is as follows. Lecture 1 (\Cref{section:HC}) provides a brief introduction to the notion of hypercyclicity, restricted to the tools required for our purposes. In particular, we establish the classical Godefroy–Shapiro criterion and recall some basic and well-known obstructions to hypercyclicity. Lecture 2 (\Cref{Section:Toeplitz}) contains a concise introduction to Hardy spaces and Toeplitz operators. Lecture 3 (\Cref{Section:HcToeplitz}) is devoted to the main question under consideration. Our discussion will focus on two fundamental results: the first one, due to Godefroy and Shapiro, characterizes hypercyclic Toeplitz operators with anti-analytic symbols; the second one, due to Shkarin, concerns tridiagonal Toeplitz operators. As will become apparent, even in this last restricted setting, the problem remains far from straightforward. The paper concludes with a brief mention of three more recent contributions to the study of hypercyclicity for Toeplitz operators.

To close this introduction, let us emphasize that the question of hypercyclicity for Toeplitz operators remains far from being completely resolved, and significant challenges persist. In particular, in \cite{FricainGrivauxOstermann_preprint}, you could find some more specific open problems. We hope that these notes will motivate interested readers to pursue further investigation into this topic.

\section{Hypercyclic operators}\label{section:HC}
\subsection{Definition of hypercyclicity}
We usually imagine chaos as something that arises in nonlinear systems, while linear systems are thought to be predictable and ``well-behaved''. But this intuition is not always correct. In fact, some natural linear operators can behave in a surprisingly chaotic way. One of their key features is the existence of a dense orbit: starting from a single point and repeatedly applying the operator, the trajectory can come arbitrarily close to every point in the space.

\begin{definition} More formally, let $X$ be a Fréchet space, that is a complete topological vector space whose topology is generated by a countable family of seminorms, and let $T:X\longrightarrow X$ be a continuous and linear operator on $X$ (in short $T\in \mathscr{L}(X))$. The operator $T$ is called \textit{hypercyclic} on $X$ if there exists a $x\in X$ such that its orbit under $T$,
\[orb(T,x)=\{T^nx\,;\,n\ge0\},\]
is dense in $X$. In such a case, $x$ is called a \textit{hypercyclic vector} for $T$ and the set of all hypercyclic vectors for $T$ is denoted by $HC(T)$.
\end{definition}
\begin{example}

The first known example of a hypercyclic operator was constructed by \textit{Birkhoff} in 1929 \cite{Birkhoff1929}, illustrating that such “chaotic” behavior can indeed occur even for linear operators. Let $Hol(\bb C)$ be the Fr\'echet space of entire functions, let $a$ be any non zero complex number and let $T_a:Hol(\bb C)\longrightarrow Hol(\bb C)$ be the associated translation operator defined by 
\[
(T_a f)(z)=f(a+z),\qquad z\in\mathbb C,\,f\in Hol(\bb C).
\]
Then $T_a$ is hypercyclic on $Hol(\bb C)$.
\end{example}

\begin{example}
Another early example of a hypercyclic operator, still on the Fréchet space $Hol(\bb C)$, was given by \textit{MacLane} in 1952 \cite{MacLane1952}. Let $D:Hol(\bb C)\longrightarrow Hol(\bb C)$ be the differentiation operator defined by $Df=f'$. Then $D$ is hypercyclic on $Hol(\bb C)$ as well.
\end{example}

\begin{example} On Banach spaces, the first known example is provided by \textit{Rolewicz} in 1969 \cite{Rolewicz1969}. Let $X=\ell^p(\bb N)$, $1\le p<+\infty$, or $X=c_0(\bb N)$. Let $T$ be the backward shift operator on $X$ defined by
    \[\begin{array}{rccc}
        T:&X&\longrightarrow& X\\
        &(x_0,x_1,\dots)&\longmapsto& (x_1,x_2,\dots)
    \end{array}\]
    Then $\lambda T$ is hypercyclic on $X$ if and only if $|\lambda|>1$.
\end{example}
Inspired by these early examples, researchers in the 1980s started studying the dynamical behavior of general linear operators, paying special attention to the phenomenon of hypercyclicity. This marked the beginning of a systematic exploration of “chaotic” behavior in linear settings. Some classical references for this area of research include:
\begin{itemize}
    \item K. G. Erdmann - A. Peris, ``Linear Chaos'', 2011
    \cite{GrosseErdmannPeris2011}; 
    \item F. Bayart - E. Matheron, ``Dynamics of linear operators'', 2009 \cite{BayartMatheron2009}.
\end{itemize}
In particular, these references contain all the results discussed in this section, and we draw on them as a source of inspiration for this brief introduction to hypercyclicity.

\subsection{Birkhoff's transitivity theorem and Godefroy Shapiro criterion} The most useful characterization of hypercyclicity is an application of the Baire category theorem and is due to Birkhoff through the notion of topological transitivity. We first recall this notion.

\begin{definition}
Let $X$ be a Fréchet space and $T\in \mathscr{L}(X)$. The operator $T$ is called \emph{topologically transitive} if for each pair of non-empty open sets $(U,V)$ of $X$, there exists $n\in\bb N$ such that $T^n(U)\cap V\neq\varnothing$.
\end{definition}

\begin{theorem}[Birkhoff's transitivity theorem, 1922 \cite{Birkhoff1922}]
Let $X$ be a separable Fréchet space and $T\in \mathscr{L}(X)$. The following assertions are equivalent.
\begin{enumerate}[(i)]
    \item $T$ is hypercyclic;
    \item $T$ is topologically transitive.
\end{enumerate}
\end{theorem}
\begin{figure}[ht]
\begin{tikzpicture}
\draw(0,0)node{\includegraphics[scale=1]{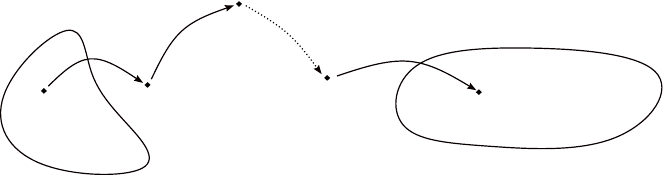}};
\draw(-4,-1)node{$U$};
\draw(4.5,-.5)node{$V$};
\draw(-4.9,.2)node{$x$};
\draw(-3.1,-.3)node{$Tx$};
\draw(-1.5,1.7)node{$T^2x$};
\draw(0,-0.15)node{$T^{n-1}x$};
\draw(3,-0.35)node{$T^nx$};
\end{tikzpicture}
    \caption{}
    \label{fig:Birkhoff}
\end{figure}
\begin{proof}
$(i)\implies (ii)$ : Assume that $T$ is hypercyclic and let $x$ be a hypercyclic vector for $T$. Consider $U$ and $V$ be two non-empty open sets of $X$. Since $orb(T,x)$ is dense in $X$, there is some $k\ge0$ such that $T^kx\in U$. But the space $X$ has no isolated points, and then the set $\{T^mx\,;\,m\ge k\}$ is also dense in $X$. Hence there exists $m\ge k$ such that $T^mx\in V$. Observe now that 
\[T^mx=T^{m-k}(T^kx)\in  T^{m-k}(U)\cap V,\]
meaning that $T^n(U)\cap V\neq\varnothing$ with $n=m-k$. Thus $T$ is topologically transitive.

$(ii)\implies (i)$ : Assume now on the contrary that $T$ is topologically transitive and let us prove that $T$ is hypercyclic. Since $X$ has a countable dense set $\{y_j\,;\,j\ge1\}$ (remember that $X$ is separable), then the open balls of radius $1/m$ around the $y_j,~j,m\ge1$, form a countable base $(V_k)_{k\ge1}$ of the topology of $X$. Hence, $x\in HC(T)$ if and only if for every $k\ge1$, there is some $n\ge0$ such that $T^nx\in V_k$. In other words
\[HC(T)=\bigcap_{k\ge1}\bigcup_{n\ge0}T^{-n}(V_k).\]
By continuity of $T$, the sets 
\[W_k=\bigcup_{n\ge0}T^{-n}(V_k)\]
are open sets for every $k\ge1$. Let's prove that $W_k$ is dense for every $k\ge1$.\\
Let $U$ be an open set of $X$. Since $T$ is topologically transitive, there exists $n\ge0$ such that \[T^n(U)\cap V_k\neq\varnothing\iff U\cap T^{-n}(V_k)\neq\varnothing.\]
Therefore $W_k$ is dense in $X$. By Baire category theorem, the set $HC(T)=\bigcap\limits_{k\ge1}W_k$ is dense in $X$ and in particular, $HC(T)\neq\varnothing$.
\end{proof}
\begin{remark}\label{Rk:EquivHcInv}
Let $T\in \mathscr{L}(X)$ and assume that $T$ is invertible. Then $T$ is hypercyclic if and only if $T^{-1}$ is hypercyclic. Indeed it is sufficient to observe that 
\[T^n(U)\cap V\neq\varnothing\iff U\cap T^{-n}(V)\neq\varnothing\]
and to apply Birkhoff theorem.
\end{remark}
As was observed by Godefroy and Shapiro, it turns out that if an operator has sufficiently eigenvectors, then it is hypercyclic.
\begin{theorem}[Godefroy - Shapiro citerion, 1991 \cite{GodefroyShapiro1991}]\customlabel{GScrit}{Godefroy - Shapiro Criterion}
Let $X$ be a separable Fréchet space and $T\in \mathscr{L}(X)$. Suppose that
\[H_-(T)=\spa\big(\ker(T-\lambda)\,;\,|\lambda|<1\big)~\text{and}~H_+(T)=\spa\big(\ker(T-\lambda)\,;\,|\lambda|>1\big)\]
are dense in $X$. Then $T$ is hypercyclic on $X$.
\end{theorem}
\begin{proof}
We shall prove that $T$ is topologically transitive. Let $U$ and $V$ be two non-empty open subsets of $X$. By assumption, $U\cap H_-(T)\neq\varnothing$ and $V\cap H_+(T)\neq\varnothing$. In other words, we can find
\[x=\sum_{k=1}^ma_kx_k\in U\quad \text{and}\quad y=\sum_{k=1}^mb_ky_k\in V\]
with $Tx_k=\lambda_kx_k$ with $|\lambda_k|<1$, $Ty_k=\mu_ky_k$ with $|\mu_k|>1$ and $a_k,b_k\in\bb C,~1\le k\le m$. 
Observe that
\[T^nx=\sum_{k=1}^ma_k\lambda_k^nx_k\xrightarrow[n \to \infty]{}0\]
and
\[u_n=\sum_{k=1}^mb_k\frac1{\mu_k^n}y_k\xrightarrow[n \to \infty]{}0\quad\text{with}\quad T^nu_n=y,~n\ge 0.\]
Taking account of the fact that $U$ and $V$ are open subsets, it follows that there is some $N\in\bb N$ so that, for all $n\ge N$,
\[x+u_n\in U\quad \text{and}\quad y+T^nx\in V.\]
In other words, $V\cap T^n(U)\neq\varnothing$ meaning that $T$ is topologically transitive. According to Birkhoff's theorem, we then deduce that $T$ is hypercyclic.
\end{proof}
\subsection{Some restrictions to hypercyclicity}
We assume now that $X$ is a \textit{complex separable Banach space}.
\begin{proposition}\label{prop:spectre-ponctuel-adjoint-vide}
Let $T\in \mathscr{L}(X)$. If $T$ is hypercyclic on $X$, then $\sigma_p(T^*)=\varnothing$.
\end{proposition}
\begin{proof}
Let $x\in HC(T)$ and argue by absurd assuming that $T^*(x^*)=\mu x^*$ for some $\mu\in\bb C$ and $x^*\in X^*,\,x^*\neq0$. Then, we have, for every $n\in \bb N$ 
\[\ps{x^*,T^nx}=\ps{T^{*n}x^*,x}=\mu^n\ps{x^*,x}.\]
Observe that the set $\{\ps{x^*,T^nx}\,;\,n\ge0\}$ is dense in $\bb C$ whereas clearly the set  $\{\mu^n\ps{x^*,x}\,;\,n\ge0\}$ is not dense. This yields the desired contradiction.
\end{proof}
\begin{proposition}\label{Prop:HCcn}
Let $T\in \mathscr{L}(X)$.
\begin{enumerate}[(a)]
    \item If $\|T\|\le1$ then $T$ is not hypercyclic on $X$.
    \item If for every $x\in X$, $\|Tx\|\ge\|x\|$, then $T$ is not hypercyclic on $X$.
\end{enumerate}
\end{proposition}
\begin{proof}
\textit{(a)} If $\|T\|\le1$, then for every $x\in X$ and every $n\ge1$, we have $\|T^nx\|\le\|x\|$. Hence the orbit $orb(T,x)$ is bounded, and then cannot be dense in $X$. Since this is true for every $x\in X$, then $T$ is not hypercyclic.

\textit{(b)} If for every $x\in X$, $\|Tx\|\ge\|x\|$, then for every $x\in X$ and every $n\ge0$, we have $\|T^nx\|\ge\|x\|$.  Hence the orbit $orb(T,x)$ stays away from $0$ (if $x\neq0$) and then cannot be dense in $X$. One more time, $T$  cannot be hypercyclic on $X$.
\end{proof}

In the context of Hilbert spaces, we will show that a hyponormal operator cannot be hypercyclic.
\begin{definition}
Let $H$ be a Hilbert space and $T\in \mathscr{L}(H)$. We say that $T$ is \textit{hyponormal} if
$T^*T-TT^*\ge0$, that is 
\begin{equation}\label{Eq1}
\forall x\in H,\quad\|Tx\|\ge\|T^*x\|.
\end{equation}
\end{definition}
\begin{fact}\label{Fact:Hypo}
If $T$ is hyponormal, then we have
\[\forall x\in H,\quad \|Tx\|^2\le\|T^2x\|\|x\|.\]
\end{fact}
\begin{proof}
Write:
\begin{align*}
\|Tx\|^2=\ps{Tx,Tx}&=\ps{T^*Tx,x}\\
&\le\|T^*(Tx)\|\|x\|&\text{by Cauchy-Schwartz inequality}\\
&\le\|T^2x\|\|x\|&\text{by \eqref{Eq1} applied to $Tx$}.
\end{align*}
\end{proof}
\begin{theorem}\label{Theorem-Hyponormal}
Let $T\in \mathscr{L}(H)$ be hyponormal. Then $T$ is not hypercyclic on $H$.
\end{theorem}
\begin{proof}
Let $x\in H,~x\neq0$.

Assume first that the sequence $(\|T^nx\|)_{n\ge0}$ is decreasing. Thus $\|T^nx\|\le\|x\|$ for all $n\ge0$ and $x$ cannot be a hypercyclic vector.

Now, it remains to prove that when the sequence $(\|T^nx\|)_{n\ge0}$ is not decreasing, then the vector $x$ cannot be hypercyclic as well. 

By assumption, there exists $N\in\bb N$ such that $\|T^{N+1}x\|>\|T^Nx\|$. In particular, it implies that $T^{N+1}x\neq0$ and then $T^Nx\neq0$. By Fact \ref{Fact:Hypo} applied to $T^Nx$, we obtain $\|T^{N+1}x\|^2\le\|T^{N+2}x\|\|T^Nx\|$, whence
    \[\|T^{N+2}x\|\ge\frac{\|T^{N+1}x\|^2}{\|T^Nx\|}=\|T^{N+1}x\|\,\underset{>1}{\underbrace{\frac{\|T^{N+1}x\|}{\|T^Nx\|}}}>\|T^{N+1}x\|.\]
    An induction shows that $(\|T^nx\|)_{n\ge N}$ is strictly increasing and in particular $\|T^nx\|>\|T^Nx\|,~\forall n\ge N+1$. Hence $x$ cannot be a hypercyclic vector.
\end{proof}

We finish this first lecture (section) by giving some spectral restrictions.
\begin{theorem}\label{Th:HcIntSp}
Let $X$ be a complex separable Banach space and $T\in \mathscr{L}(X)$. Assume that $T$ is hypercyclic. Then $\sigma(T)\cap\bb T\neq\varnothing$.
\end{theorem}
\begin{proof}
We argue by absurd and suppose that $\sigma(T)\cap\bb T=\varnothing$. 

If $\sigma(T)\subset\bb D$ then by the spectral radius formulae $r(T)<1$ and there exists $\varepsilon>0$ and $M>0$ such that 
\[\|T^n\|\le M(1-\varepsilon)^n\xrightarrow[n \to \infty]{}0,\]
which contradicts the fact that $T$ is hypercyclic.

If $\sigma(T)\subset\bb C\setminus\overline{\bb D}$, then $T$ is invertible and $T^{-1}$ is also hypercyclic. But since $\sigma(T^{-1})=\sigma(T)^{-1}\subset\bb D$, we also get a contradiction.

Therefore $\sigma_1=\sigma(T)\cap\bb D$ and $\sigma_2=\sigma(T)\cap\big(\bb C\setminus\overline{\bb D}\big)$ form a partition of $\sigma(T)$ into two non empty closed sets. By the Riesz decomposition theorem, there exists two non-trivial closed invariant subspaces $M_1$ and $M_2$ such that 
\[X=M_1\oplus M_2\quad\text{and}\quad\sigma(T_{|M_j})=\sigma_j,\,1\le j\le 2.\]
But it is not difficult to see that $T_{|M_1}$ should be hypercyclic, which is again a contradiction because $\sigma(T_{|M_1})=\sigma_1\subset\bb D$.
\end{proof}
Kitai proved indeed a better result.
\begin{theorem}[Kitai, 1982 \cite{Kitai1982}]
Let $X$ be a complex separable Banach space and $T\in \mathscr{L}(X)$. Assume that $T$ is hypercyclic. Then every connected component of $\sigma(T)$ intersects $\bb T$.
\end{theorem}
\section{Toeplitz operators and Hardy spaces}\label{Section:Toeplitz}
\subsection{Toeplitz matrices}
Recall that a Toeplitz matrix on $\ell^2(\bb N)$ is a matrix of the form
\[T(a)~=~
\begin{pmatrix}
 a_0    & a_{-1} & a_{-2} & a_{-3} & \cdots \\
 a_1    & a_0    & a_{-1} & a_{-2} & \cdots \\
 a_2    & a_1    & a_0    & a_{-1} & \cdots \\
 a_3    & a_2    & a_1    & a_0    & \cdots \\
 \vdots & \vdots & \vdots & \vdots & \ddots
\end{pmatrix}
\]
where $a=(a_n)_{n\in\bb Z}$ is a doubly infinite sequence of complex numbers. In other words $A=(a_{j,k})_{j,k\ge0}$ is a Toeplitz matrix $T(a)$ if $a_{j,k}=a_{j-k}$, for every $j,k\ge0$. To study boundedness of a Toeplitz matrix on $\ell^2(\bb N)$, it will be useful to adopt another point of view. 

Let $L^2(\bb T)$ be the Hilbert space of Lebesgue measurable function on the unit circle $\bb T$ equipped with the inner product
\[\ps{f,g}_2~=~\int_0^{2\pi}f(e^{i\theta})\overline{g(e^{i\theta})}\frac{\mathrm d\theta}{2\pi}.\]
Let $H^2=\left\{f\in L^2(\bb T)\,;\,\hat f(n)=0,\,\forall n<0\right\}$ be the Hardy space. It turns out that $H^2$ is a closed subspace of $L^2(\bb T)$ and we denote by $P_+:L^2(\bb T)\longrightarrow H^2$ the orthogonal projection of $L^2(\mathbb T)$ onto $H^2$, defined by
\[P_+\left(\sum_{n\in\bb Z}c_ne^{in\theta}\right)=\sum_{n\ge0}c_ne^{in\theta}.\]
Note that 
\begin{equation}\label{ker-Riesz-projection}
\mbox{ker}(P_+)=\overline{H^2_0},
\end{equation}
where $H^2_0=e^{i\theta}H^2$.
Moreover, the operator
    \[\begin{array}{rccc}
        U:~&\ell^2(\bb N)&\longrightarrow&H^2\\
        &(a_n)_{n\ge0}&\longmapsto& \displaystyle\sum_{n\ge0}a_ne^{in\theta}
    \end{array}\]
is a unitary operator.

Hartman and Wintner in 1954 got a characterization of the Toeplitz matrices that give rise to bounded operators on $\ell^2(\bb N)$.
\begin{theorem}[Hartman - Wintner, 1954 \cite{HartmanWintner1954}]
The following assertions are equivalent:
\begin{enumerate}[(i)]
    \item The operator $T(a)$ is bounded on $\ell^2(\bb N)$;
    \item There exists a function $\phi\in L^\infty(\bb T)$ such that $a_n=\widehat{\phi}(n),~\forall\,n\in\bb Z$.
\end{enumerate}
Moreover, $UT(a)U^{-1}=T_\phi$, where $T_\phi f=P_+(\phi f)\,f\in H^2$.
\end{theorem}
\begin{proof}[Proof of $(ii)\implies(i)$]
First note that if $\phi\in L^\infty(\bb T)$, then, for every $f\in H^2$, we have
\[\|T_\phi f\|_2=\|P_+(\phi f)\|_2\le\|\phi f\|_2\le\|\phi\|_\infty\|f\|_2.\]
Hence $T_\phi$ is bounded on $H^2$ and $\|T_\phi\|\le\|\phi\|_\infty$. 

Moreover, let $(e_n)_{n\ge0}$ be the canonical basis of $\ell^2(\bb N)$. Then, for every $n,k\ge 0$, on one hand,  we have $\ps{T(a)e_n,e_k}=a_{k-n}$,  and on the other hand, we have
\begin{align*}
\ps{U^{-1}T_\phi Ue_n,e_k}~&=~\ps{T_\phi Ue_n,e_k}\\
&=~\ps{T_\phi z^n,z^k}\\
&=~\ps{P_+(\phi z^n),z^k}\\
&=~\ps{\phi z^n,z^k}&\text{because }z^k\in H^2\\
&=~\ps{\phi,z^{k-n}}\\
&=~\widehat{\phi}(k-n)~=~a_{k-n}.
\end{align*}
Hence, for every $n,k\ge0$, we have $\ps{T(a)e_n,e_k}=\ps{U^{-1}T_\phi Ue_n,e_k}$, which proves that $T(a)=U^{-1}T_\phi U$ is bounded on $\ell^2(\bb N)$.

For the proof of $(i)\implies(ii)$, we refer to the following book: ``An introduction to operators on the Hardy-Hilbert space'' by Rubén A. Martinez-Avenda\~no and Peter Rosenthal, 2007 \cite{MartinezAvendanoRosenthalB}.
\end{proof}

We will now discuss basic properties of Toeplitz operators $T_\phi$ on $H^2$ when the symbol $\phi\in L^\infty(\bb T)$. Before, we recall some basic facts on $H^2$.
\subsection{The Hardy space $H^2$}
The space $H^2$ can be identified with the space $H^2(\bb D)$ which consists of analytic functions $f$ on the open unit disc $\bb D$, 
\[f(z)~=~\sum_{n=0}^{+\infty} a_nz^n,\qquad z\in\mathbb D,\]
where Taylor coefficients $(a_n)_{n\ge0}\in\ell^2(\bb N)$. The space $H^2(\bb D)$ is equipped with the following scalar product
\begin{equation}\label{defn-inner-product-hardyspace}
\ps{f,g}_2~=~\sum_{n=0}^{+\infty}a_n\overline{b_n},
\end{equation}
where $f(z)=\sum_{n=0}^\infty a_n z^n$ and $g(z)=\sum_{n=0}^\infty b_n z^n$. Then the operator
    \[\begin{array}{rccc}
        U:~&H^2(\bb D)&\longrightarrow&H^2\\
        &\displaystyle f(z)=\sum_{n=0}^\infty a_n z^n&\longmapsto& \displaystyle f^*(e^{i\theta})=\sum_{n=0}^\infty a_n e^{in\theta}
    \end{array}\]
    is a unitary operator and $\widehat{f^*}(n)=a_n,\,n\ge0$.

Using Parseval equality, it is not difficult to see that if $f_r(e^{i\theta})=f(re^{i\theta})$, $0\leq r<1$, $\theta\in\bb R$, then 
\[
\lim_{r\to 1^-}\|f_r-f^*\|_{L^2(\mathbb T)}= 0.
\]
In particular, there exists a sequence $(r_n)_n\subset (0,1)$ such that $r_n\to 1$ as $n\to \infty$ and 
\[
\lim_{n\to\infty}f(r_ne^{i\theta})=f^*(e^{i\theta}),\qquad 
\text{a.e. }e^{i\theta}\in\bb T.
\]
In fact, a deep result of Fatou says much more: if $f\in H^2(\bb D)$, then
    \[nt-\!\!\lim_{z\to e^{i\theta}}f(z)~=~f^*(e^{i\theta})\qquad \text{a.e. }e^{i\theta}\in\bb T,\]
    where $nt-\lim$ denotes the non-tangential limit restricted to a Stolz region $\Delta_\alpha$.

     \begin{figure}[ht]
        \begin{tikzpicture}
        \draw(0,0)arc[start angle=-30,end angle=30,radius=2.5];
        \draw(.335,1.25)node{$\times$}(.335,1.25)node[right]{$e^{i\theta}$};
        \draw[blue,fill=cyan!20](-1,.5)--(.335,1.25)--(-1,2);
        \draw[blue,->](0,1.45)arc[start angle=160,end angle=200,radius=.57];
        \draw[blue](-.2,1.25)node{$\alpha$};
        \draw(2,1.25)node[right]{with $0<\alpha<\pi$.};
        \end{tikzpicture}
        \caption{}
        \label{Fig2}
    \end{figure}   
    
In the sequel, we will identify, as usual, $H^2$ and $H^2(\bb D)$ and omit the star in $f^*$ when we talk about the non-tangential limit of a function in $H^2(\bb D)$. The space $H^2$ is a RKHS (\emph{reproducing kernel Hilbert space}) where the kernel is given by
\[k_\lambda(z)~=~\frac1{1-\overline{\lambda}z},\quad\lambda, z\in\bb D.\]
Indeed, we have $k_\lambda(z)=\sum\limits_{n\ge0}\overline{\lambda}^nz^n,$ and $(\overline{\lambda}^n)_{n\ge0}\in \ell^2(\bb N)$. Hence the function $k_\lambda$ belongs to $H^2$. Moreover, for every $f\in H^2,\, f(z)=\sum\limits_{n\ge0}a_nz^n$, and for every $\lambda\in\bb D$, taking account of \eqref{defn-inner-product-hardyspace}, we have
\[\ps{f,k_\lambda}~=~\sum_{n=0}^\infty a_n\lambda^n~=~f(\lambda).\]
\subsection{Toeplitz operators on $H^2$}
If we come back to Toeplitz operators on $H^2$, we have the following result on the norm.
\begin{theorem}[Brown - Halmos, 1963 \cite{BrownHalmos1963}]
Let $\phi\in L^\infty(\bb T)$. Then $T_\phi\in \mathscr{L}(H^2)$ and $\|T_\phi\|=\|\phi\|_\infty.$
\end{theorem}
\begin{proof}
We have already seen that $T_\phi$ is bounded and $\|T_\phi\|\le\|\phi\|_\infty$.\\ For the converse inequality, denote by $\widetilde{k_\lambda}=k_\lambda/\|k_\lambda\|_2$ the normalized reproducing kernel. We have
\[\left|\ps{T_\phi\widetilde{k_\lambda},\widetilde{k_\lambda}}_2\right|~\le~\|T_\phi\widetilde{k_\lambda}\|_2\|\widetilde{k_\lambda}\|_2~\le~\|T_\phi\|.\]
Moreover, using that $\|k_\lambda\|^2_2=\ps{k_\lambda,k_\lambda}_2=k_\lambda(\lambda)=\frac{1}{1-|\lambda|^2}$, we have
\begin{align*}
\ps{T_\phi\widetilde{k_\lambda},\widetilde{k_\lambda}}_2~
&=~\ps{P_+(\phi\widetilde{k_\lambda}),\widetilde{k_\lambda}}_2\\
&=~\ps{\phi\widetilde{k_\lambda},\widetilde{k_\lambda}}&\text{because }\widetilde{k_\lambda}\in H^2\\
&=~\int_{\bb T}\phi(e^{i\theta})\frac{|k_\lambda(e^{i\theta})|^2}{\|k_\lambda\|_2^2} \,\mathrm dm(e^{i\theta})\\
&=~\int_{\bb T}\phi(e^{i\theta})\underset{\text{Poisson kernel}}{\underbrace{\frac{1-|\lambda|^2}{|e^{i\theta}-\lambda|^2}}} \,\mathrm dm(e^{i\theta})~=~\mathcal P(\phi)(\lambda),
\end{align*}
where $\mathcal P(\phi)$ is the Poisson transform of $\phi$. 
Hence we get that for every $\lambda\in\bb D$, we have
\begin{equation}\label{eq:poisson-norm}
|\mathcal P(\phi)(\lambda)|~\le~\|T_\phi\|.
\end{equation}
We will now use a well-known property of the Poisson integral (see \cite[Corollary 1.1.27]{MartinezAvendanoRosenthalB}) saying that 
\[
\lim_{r\to 1^-}\mathcal P(\phi)(r e^{i\theta})=\phi(e^{i\theta})\qquad \text{a.e. }e^{i\theta}\in\bb T
\]
Hence, according to \eqref{eq:poisson-norm}, it follows  that for almost all $e^{i\theta}\in\bb T$, we have 
\[|\phi(e^{i\theta})|\le\|T_\phi\|.\]
Hence $\|\phi\|_\infty\le\|T_\phi\|$, which concludes the proof.
\end{proof}
\noindent\textbf{Few observations. }Let $\phi\in L^\infty$. It is easy to see that 
\begin{enumerate}[(a)]
    \item $T_\phi^*=T_{\overline{\phi}}.$
    In particular, if $S=T_z$ is the shift operator on $H^2$, then $S^*=T_{\bar z}$ is the backward shift operator on $H^2$. For every $f\in H^2$, we have
\[Sf(z)~=~zf(z)\quad\text{and}\quad S^*f(z)~=~\frac{f(z)-f(0)}{z}.\]
\item If $\phi\in H^\infty$, the algebra of bounded and analytic functions on $\bb D$, then, for every $\lambda\in \bb D$, we have 
\begin{equation}\label{eq:Toeplitz-vector-co-analytic}
T_{\overline{\phi}}k_\lambda~=~\overline{\phi(\lambda)}k_\lambda.
\end{equation}
Indeed, let $f\in H^2$, then
\[
\ps{f,T_{\overline{\phi}}k_\lambda}_2~=~\ps{f,T_{\phi}^*k_\lambda}_2~=~\ps{T_\phi f,k_\lambda}_2~=~\ps{P_+(\phi f),f_\lambda}_2~=~\ps{\phi f,k_\lambda}_2.
\]
Now observe that since $\phi\in H^\infty$, then $\phi f\in H^2$ and we deduce that
\[
\ps{f,T_{\overline{\phi}}k_\lambda}_2~=~(\phi f)(\lambda)~=~\phi(\lambda)f(\lambda)~=~\ps{f,\overline{\phi(\lambda)}k_\lambda}_2.
\]
Therefore, $T_{\overline{\phi}}k_\lambda~=~\overline{\phi(\lambda)}k_\lambda$.\qed
\item In general, $T_\phi T_\psi\neq T_{\phi\psi}$ but Brown and Halmos proved the following.
\end{enumerate}
\begin{theorem}[Brown - Halmos, 1963 \cite{BrownHalmos1963}]
Let $\phi,\psi\in L^\infty(\bb T)$ Then
\[
T_\phi T_\psi~=~ T_{\phi\psi}\quad\iff\quad \text{either }\overline{\phi}\text{ or }\psi\text{ is in }H^\infty.
\]
\end{theorem}
We will do not really use this result so we will admit it. However, note that, in general, we have $T_\phi^n\neq T_{\phi^n}$ and there is no tracktable formula for $T_\phi^n$, which makes the question of hypercyclicity for general Toeplitz operators very difficult.

As we have seen in the first part, the question of eigenvectors is crucial in the study of hypercyclicity. We will now give two results in this direction.
\begin{theorem}[Coburn, 1966 \cite{Coburn1966}]
Let $\phi\in L^\infty (\bb T),\phi\not\equiv0$. Then either $\ker T_\phi=\{0\}$ or $\ker T_\phi^*=\{0\}$.
\end{theorem}
\begin{proof}
Argue by absurd and assume that there are two functions $h_1$ and $h_2$ in $H^\infty,h_1,h_2\neq0$ such that $T_\phi h_1=0$ and $T_{\overline{\phi}}h_2=0$. Taking account of \eqref{ker-Riesz-projection}, it means that $\phi h_1\in\overline{H^2_0}$ and $\overline{\phi}h_2\in \overline{H^2_0}$. It is an easy exercise to prove that if $f\in H^2_0$ and $g\in H^2$, then $fg\in L^1(\bb T)$ and $\widehat{fg}(n)=0$, $\forall n\le0$. 

Let now define $\psi=h_1\phi\overline{h_2}$. Since $\phi \overline{h_2}\in H_0^2$ and $h_1\in H^2$, then $\psi\in L^1(\bb T)$ and $\widehat{\psi}(n)=0$, $\forall n\le0$. Moreover $\overline{\psi}= \overline{h_1\phi} h_2$, with $\overline{h_1\phi}\in H_0^2$ and $h_2\in H^2$. Hence $\widehat{\,\overline{\psi}\,}(n)=0$ , $\forall n\le0$, which give that $\widehat{\psi}(n)=0$, $\forall n\ge0$. Therefore $\widehat{\psi}(n)=0$, $\forall n\in\bb Z$ and thus $\psi\equiv0$.
But a standard property of Hardy spaces (see \cite[Corollary 2.7.2]{MartinezAvendanoRosenthalB}) shows that, since $h_1,h_2\neq0$, then $h_1$ and $h_2$ are non-zero almost everywhere. Hence $\phi\equiv0$ a.e. which is a contradiction.
\end{proof}
When the symbol $\phi$ is continuous, we can link the spectral properties of $T_\phi$ with the winding number of $\phi$. Recall that if $\phi\in C(\bb T)$ and $\lambda\notin \phi(\bb T)$, then its winding number is defined as 
\begin{align*}
\w_\phi(\lambda)~
&=~\frac1{2i\pi}\int_{\phi(\bb T)}\frac1{u-\lambda}\,\mathrm du\\
&=~\frac1{2i\pi}\int_{\bb T}\frac{\phi'(\zeta)}{\phi(\zeta)-z}\,\mathrm d\zeta&\text{if }\phi\in C^1(\bb T).
\end{align*}
Furthermore, recall that an operator $T\in \mathscr{L}(H)$ is said to be \textit{Fredholm} is:
\begin{itemize}
    \item[(i)] $\Im(T)$ is closed;
    \item[(ii)] $\ker(T)$ and $\ker(T^*)$ are both finite dimensional.
\end{itemize}
In this case, the index of $T$ is defined as
\[j(T)=\dim(\ker T)-\dim(\ker T^*).\]
When $\phi$ is continuous, we have the following description of the spectrum.
\begin{theorem}
Let $\phi\in C(\bb T)$.
\begin{enumerate}[(a)]
    \item $T_\phi$ is a Fredholm operator on $H^2$ if and only if $\phi$ does not vanish on $\bb T$. In this case, 
    \[j(T_\phi)~=~-\w_\phi(0).\]
    \item We have
    \[\sigma(T_\phi)~=~\phi(\bb T)\cup\{\lambda\in\bb C\setminus\phi(\bb T)\,;\,\w_\phi(\lambda)\neq0\}.\]
\end{enumerate}
\end{theorem}
This beautiful theorem is the combined work of several mathematicians: Krein, Calder\'on, Spitzer, Widom, Devinatz. Since we will do not really us this result in the following, we will admit it. However, we refer to the book of Martinez-Avenda\~no - Rosenthal \cite{MartinezAvendanoRosenthalB} for its proof and more details. 

Observe that if $\lambda\in \bb C\setminus\phi(\bb T)$ with $\w_\phi(\lambda)<0$, then 
\[
-\w_\phi(\lambda)=j(T_{\phi-\lambda})=\dim\ker(T_\phi-\lambda)-\dim\ker(T_\phi^*-\bar\lambda)
\]
so, we necessary have $\dim\ker(T_\phi-\lambda)>0$. Hence
\[
\{\lambda\in\bb C\setminus \phi(\bb T)\,;\,\w_\phi(\lambda)<0\}\subset\sigma_p(T_\phi),
\]
and 
\begin{equation}\label{eq:windnumber-spectre}
\{\lambda\in\bb C\setminus \phi(\bb T)\,;\,\w_\phi(\lambda)>0\}\subset\sigma_p(T_\phi^*).
\end{equation}
\begin{example}~
\begin{enumerate}
 \item If $S$ denotes the shift operator on $H^2$, it can be easily proved that $\sigma(S^*)=\overline{\bb D}$ and $\sigma_p(S^*)=\bb D$.
    \item Let $\phi$ be a continuous symbol such that $\phi(\mathbb T)$ is represented as follows:
    \begin{figure}[ht]
\begin{tikzpicture}
\draw(0,0)node{\includegraphics[scale=.8]{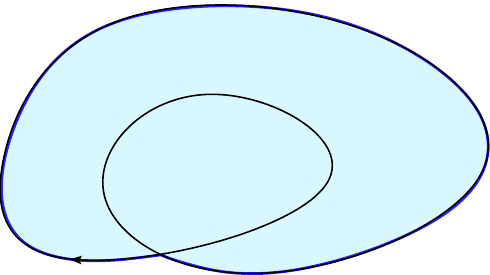}};
\draw(0,0)node{$-2$};
\draw(1,1)node{$-1$};
\draw(2.5,1.5)node{$0$};
\draw(-2,-1.4)node{$\phi(\bb T)$};
\draw[blue,<-](1,-1.5)to[bend left](3,-1.5)node[right]{$\sigma(T_\phi)$};
\end{tikzpicture}
        \caption{}
        \label{Fig3}
    \end{figure}
   
\end{enumerate}
\end{example}

\section{Hypercyclicity of Toeplitz operators}\label{Section:HcToeplitz}
Remind that we are interested in the following question: given a function $\phi\in L^\infty(\mathbb T)$, where $\mathbb T=\partial \mathbb D$ denotes the unit circle equipped with the normalized Lebesgue measure $m$, can one determine some necessary and/or sufficient conditions on $\phi$ to ensure that the associated Toeplitz operator $T_\phi$ is hypercyclic on $H^2$? As already said in the introduction, we will focus on two key results where we have a complete characterization. The first one concerns the case where the symbol $\phi$ is anti-analytic (meaning that $\phi\in \overline{H^\infty}$). The second one concerns the case where the symbol $\phi$ is of the form $\phi(z)=a/z+b+cz$, $z\in \mathbb T$ (which corresponds to tridiagonal Toeplitz matrices). 
\subsection{Analytic and anti-analytic symbols}
We start with a simple observation. If $\phi\in H^\infty$, then $T_\phi$ is not hypercyclic on $H^2$. Indeed, since $\phi\in H^\infty$, then for every $\lambda\in \bb D$, we have
\[T_\phi^*k_\lambda\,=\,T_{\bar\phi}k_\lambda\,=\,\overline{\phi(\lambda)}k_\lambda,\]
(see \eqref{eq:Toeplitz-vector-co-analytic}) and in particular $\sigma_p(T_\phi^*)\neq\varnothing$. Thus $T_\phi$ cannot be hypercyclic by Proposition~\ref{prop:spectre-ponctuel-adjoint-vide}. 

The case of anti-analytic symbol was studied by Godefroy and Shapiro.
\begin{theorem}[Godefroy - Shapiro, 1991 \cite{GodefroyShapiro1991}] Let $\phi\in H^\infty$. The following assumptions are equivalent.
\begin{enumerate}[(i)]
    \item $T_{\bar \phi}$ is hypercyclic on $H^2$;
    \item $\phi$ is non-constant and $\phi(\bb D)\cap\bb T\neq\varnothing$.
\end{enumerate}
\end{theorem}
\begin{proof}
$(ii)\implies (i)$. Since $\phi$ is non constant, by the open mapping theorem, $\phi(\bb D)$ is an open connected set which intersects $\bb T$.

\begin{figure}[ht]
\begin{tikzpicture}
 \draw(0,0)node{\includegraphics[scale=.7]{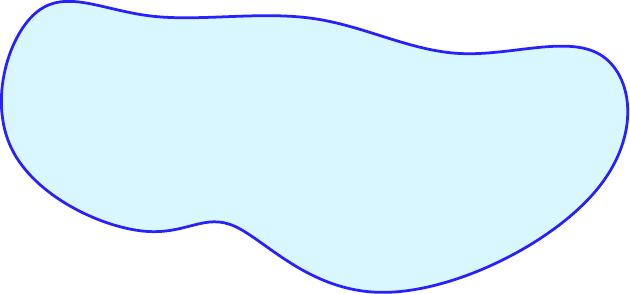}};
 \draw[blue](0,0)node{$\phi(\bb D)$};
 \draw(3,-1)arc(0:360:1);
 \draw(3.1,-1)node[below]{$\bb T$};
\end{tikzpicture}
    \caption{}
    \label{fig4}
\end{figure}

Let $U=\{z\in\bb D\,;\,|\phi(z)|<1\}$ and $V=\{z\in\bb D\,;\,|\phi(z)|>1\}$. Then $U$ and $V$ are two open sets which are both non empty. On the other hand, it follows from \eqref{eq:Toeplitz-vector-co-analytic} that
\[\spa\{k_z\,;\,z\in U\}\subset H_-(T_{\bar\phi})\quad\text{and}\quad \spa\{k_z\,;\,z\in V\}\subset H_+(T_{\bar\phi}),\]
where we recall that 
\[H_-(T)=\spa\big(\ker(T-\lambda)\,;\,|\lambda|<1\big)~\text{and}~H_+(T)=\spa\big(\ker(T-\lambda)\,;\,|\lambda|>1\big)\]
According to \ref{GScrit}, it is sufficient to prove that if $\Omega$ is a non empty open subset of $\bb D$, then $\spa\{k_z\,;\,z\in\Omega\}$ is dense in $H^2$. Thus let $f\in H^2$ such that $f\perp k_z,\,\forall\,z\in\Omega$. Hence
\[f(z)~=~\ps{f,k_z}~=~0,\quad\text{for every }z\in\Omega.\]
But since $\Omega$ is a non-empty open subset of $\bb D$, the uniqueness principle implies that $f\equiv0$. Hence $H_-(T_{\bar\phi})$ and $H_+(T_{\bar\phi})$ are both dense in $H^2$, which implies that $T_{\bar \phi}$ is hypercyclic on $H^2$.

$(i)\implies(ii)$. Assume now that $T_{\bar \phi}$ is hypercyclic on $H^2$. Then, of course, $\phi$ is non-constant and $T_{\bar\phi}$ is not a contraction. In particular, $\|\phi\|_\infty=\|T_\phi\|=\|T_{\bar\phi}\|>1$. Moreover, we also have $\inf\limits_{z\in\bb D}|\phi(z)|<1$. Indeed, if $\inf\limits_{z\in\bb D}|\phi(z)|\ge1$, then $1/\phi\in H^\infty$ and $T_{1/\phi}^*$ is not hypercyclic because $\|T_{1/\phi}^*\|=\|T_{1/\phi}\|=\|1/\phi\|_{\infty}\le1$. But since $T_{1/\phi}T_\phi=T_\phi T_{1/\phi}=I$, the operator $T_\phi$ is invertible and $T_\phi^{-1}=T_{1/\phi}$. Thus $(T_\phi^*)^{-1}=T_{1/\phi}^*$. Now, since $T_{1/\phi}^*$ is not hypercyclic, it follows from \Cref{Rk:EquivHcInv} that $T_\phi^*=T_{\bar\phi}$ is not hypercyclic as well, which is absurd. Hence we have
\[\inf_{z\in\bb D}|\phi(z)|\,<\,1\,<\,\sup_{z\in\bb D}|\phi(z)|.\]
Now, a simple connectedness argument implies that $\phi(\bb D)\cap\bb T\neq\varnothing$.
\end{proof}
We immediately recover the result of Rolewicz.
\begin{corollary}[Rolewicz, 1969 \cite{Rolewicz1969}]
Let $\lambda\in \bb C$.Then 
\[\lambda S^*=T_{\lambda\bar z}\text{ is hypercyclic on }H^2\iff|\lambda|>1.\]
\end{corollary}
\subsection{Tridiagonal Toeplitz operators}
In \cite{Shkarin2012}, S. Shkarin characterized hypercyclic Toeplitz operators $T_F$ with symbols of the form 
\[F(e^{i\theta})\,=\,ae^{-i\theta}+b+ce^{i\theta},\quad a,b,c\in\bb C,\,c\neq0.\]
They correspond to the following Toeplitz matrices which are tridiagonal.
\[
\begin{pmatrix}
 ~b&a&&&&&\\
 ~c&b&a&&(0)&\\
 &c&b&a&&\\
 &&\ddots&\ddots&\ddots&\\
&(0) &&\ddots&\ddots&\ddots\\
\end{pmatrix}
\]
Observe that $T_F=aS^*+bI+cS$ and then $T_F^*=\bar a S+\bar b I+\bar cS^*$. Using that $S^*S=I$, we get
\begin{align*}
    T_F^*T_F-T_FT_F^*\,
    &=\,(\bar aS+\bar bI+\bar cS^*)(aS^*+bI+cS)-(aS^*+bI+cS)(\bar aS+\bar bI+\bar cS^*)\\
    &=\, (|a|^2-|c|^2)SS^*+|c|^2-|a|^2\\
    &=\, (|c|^2-|a|^2)(I-SS^*).
\end{align*}
Since $S^*$ is a contraction, we have $I-SS^*\ge0$. Thus if $|c|\ge|a|$, we get $T_F^*T_F-T_FT_F^*\ge0$ meaning that $T_F$ is hyponormal. In particular, if $|c|\ge|a|$ then $T_F$ cannot be hypercyclic (see \Cref{Theorem-Hyponormal}). It is also possible to check that if $|c|>|a|$, then there exists $\lambda\in\mathbb C\setminus F(\mathbb T)$, such that
$\w_F(\lambda)>0$ (see \cite[Fact 5.3]{FricainGrivauxOstermann_preprint}). Then, according to \eqref{eq:windnumber-spectre}, we get $\sigma_p(T_F^*)\neq\varnothing$, which implies, by \Cref{prop:spectre-ponctuel-adjoint-vide}, that $T_F$ cannot be hypercyclic on $H^2$.  Note that this last argument extends to the Banach setting $H^p$ but not to the case $|a|=|c|$.

Thus we may assume that $|c|<|a|$. We need now two lemmas.
\begin{lemma}\label{Le:SkL1}
Let $F(z)=\frac az+b+cz,\,|a|>|c|>0$. Then
\begin{enumerate}[(a)]
    \item $F(\bb T)$ is an ellipse;
    \item The interior $\mathcal E$ of the ellipse satisfies
    \[\mathcal E=F\left(\left\{z\in\bb C\,;\,1<|z|<\left|\frac ac\right|\right\}\right).\]
\end{enumerate}
\end{lemma}
\begin{figure}[ht]
\begin{tikzpicture}
 \draw(0,0)node{\includegraphics[scale=.75]{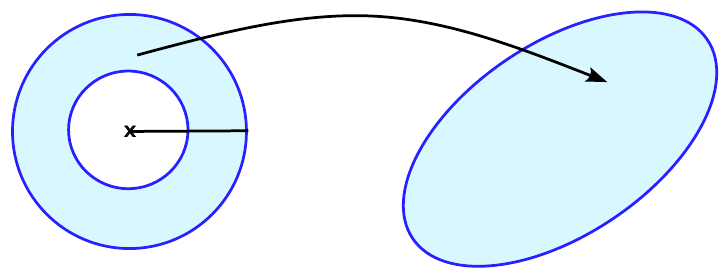}};
 \draw(0,1.8)node{$F$};
 \draw[blue](3,0)node{$\mathcal E$};
 \draw[blue](3.9,-1.15)node{$F(\bb T)$};
 \draw(-2.1,-.15)node{$1$}(-1.2,0)node{$\left|\frac ac\right|$};
 \draw(-3,-.9)node{$\bb T$};
\end{tikzpicture}
    \caption{}
    \label{fig5}
\end{figure}
\begin{proof}
Let $a=|a|e^{i\alpha},\,c=|c|e^{i\gamma}$ and consider $\widetilde F(z)=\frac{|a|}z+|c|z$. Then
\begin{align*}
F(z)\,
&=\, b+\frac{|a|}ze^{i\alpha}+|c|e^{i\gamma}z\\
&=\,b+e^{i\left(\frac{\alpha+\gamma}2\right)}\left(\frac{|a|}ze^{i\left(\frac{\alpha-\gamma}2\right)}+|c|e^{i\left(\frac{\gamma-\alpha}2\right)}z\right)\\
&=\, b+e^{i\left(\frac{\alpha+\gamma}2\right)}\widetilde F\left(e^{i\left(\frac{\gamma-\alpha}2\right)}z\right).
\end{align*}
This formula shows that \textit{we can assume, without loss of generality, that $b=0$ and $a$ and $c$ are real with $a>c>0$}.

$(a)$ $F(e^{i\theta})=ae^{-i\theta}+ce^{i\theta}=(a+c)\cos\theta+i(c-a)\sin\theta$. Hence
\[F(\bb T)\,=\,\left\{\,x+iy\in\bb C\,;\,\frac{x^2}{(a+c)^2}+\frac{y^2}{(a-c)^2}=1\,\right\}\]
and thus $F(\bb T)$ is an ellipse centered at $0$ whose foci are $(\pm2\sqrt{ac},0)$ and the length of the axis are $2(a+c)$ and $2(a-c)$.

$(b)$ We have
\[\mathcal E\,=\,\left\{\,x+iy\in\bb C\,;\,\frac{x^2}{(a+c)^2}+\frac{y^2}{(a-c)^2}<1\,\right\}.\]
Let us first check that $F\left(\left\{z\in\bb C\,;\,1<|z|<\left|\frac ac\right|\right\}\right)\subset\mathcal E$. Let $z=re^{i\theta},\,1<r<a/c$. Thus
\begin{equation}\label{Eq2}
F(re^{i\theta})=\frac are^{-i\theta}+cre^{i\theta}=\left(\frac ar+cr\right)\cos\theta+i\left(cr-\frac ar\right)\sin\theta.
\end{equation}
If we denote by $\phi(r)=\frac ar+cr$, then it is easy to see that
\begin{center}
    \begin{tikzpicture}
  \tkzTabInit[lgt=2,espcl=4]
    {$r$ /1, $\phi'(r)$ /1, $\phi(r)$ /2}
    {$1$, $\sqrt{\frac ac}$, $\frac ac$}
  \tkzTabLine{,-,z,+,}
  \tkzTabVar{+/ $a+c$, -/ $2\sqrt{ac}$, +/ $a+c$}
\end{tikzpicture}
\end{center}
Hence for all $1<r<\frac ac$, we have $0<\frac ar+cr<a+c$. Similarly, if we denote by $\psi(r)=cr-\frac ar$, then 
\begin{center}
    \begin{tikzpicture}
  \tkzTabInit[lgt=2,espcl=4]
    {$r$ /1, $\psi'(r)$ /1, $\psi(r)$ /2}
    {$1$, $\sqrt{\frac ac}$, $\frac ac$}
  \tkzTabLine{,,+,,}
\draw[->](2.5,-3.5)node[below]{$c-a$}--(10.5,-2.5)node[above]{$c+a$};
\draw(6.5,-3)node{$0$};
\end{tikzpicture}
\end{center}
Hence for all $1<r<\frac ac$, we have $\left| cr-\frac ar\right|<a-c$. Thus
\[\frac{\left(\frac ar+cr\right)^2\cos^2\theta}{(a+c)^2}+\frac{\left(cr-\frac ar\right)^2\sin^2\theta}{(a-c)^2}<\cos^2\theta+\sin^2\theta=1,\]
which means that $F(r^{i\theta})\in\mathcal E$. 

Let us now check that $\mathcal E\subset F\left(\left\{z\in\bb C\,;\,1<|z|<\left|\frac ac\right|\right\}\right).$ Let $x+iy\in\mathcal E$ then $\frac{x^2}{(a+c)^2}+\frac{y^2}{(a-c)^2}<1$. According to \eqref{Eq2}, we need to prove that there exists $r_0$ with $1<r_0<a/c$ and $\theta\in\bb R$ such that
\[
\begin{cases}
x&=~\displaystyle\left(\frac a{r_0}+cr_0\right)\cos\theta;\\[.3cm]
y&=~\displaystyle\left(cr_0-\frac a{r_0}\right)\sin\theta.
\end{cases}
\]
Let $\phi_1(r)=\frac{x^2}{\left(\frac ar+cr\right)^2}+\frac{y^2}{\left(cr-\frac ar\right)^2}$ for $r\neq\sqrt{\frac ac}$. The function $\phi_1$ is continuous and differentiable on $\left(1,\frac ac\right)\setminus\left\{\sqrt{\frac ac}\right\}$ and we have
\begin{align*}
\phi_1'(r)\,
&=\,\dfrac{~\displaystyle -2x^2\left(c-\frac a{r^2}\right)~}{~\displaystyle \left(\frac ar+cr\right)^3~}\,-\,\dfrac{~\displaystyle 2y^2\left(c+\frac a{r^2}\right)~}{~\displaystyle \left(cr-\frac ar\right)^3~}\\
&=\,\dfrac{\displaystyle~   -\,2x^2\left(c-\frac a{r^2}\right)\left(cr-\frac ar\right)^3\,-\,2y^2\left(c+\frac a{r^2}\right)\left(\frac ar+cr\right)^3   ~}{\displaystyle~   \left(\frac ar+cr\right)^3\left(cr-\frac ar\right)^3   ~}\\
&=\,-\,\dfrac{\displaystyle~ 2x^2r^3\left(c-\frac a{r^2}\right)^4\,+\,2y^2\left(c+\frac a{r^2}\right)\left(\frac ar+cr\right)^3   ~}{\displaystyle~   \left(\frac ar+cr\right)^3\left(cr-\frac ar\right)^3   ~}.
\end{align*}
Observe that $\phi_1'(r)\ge0\iff cr-\frac ar\le0\iff r<\sqrt{\frac ar}$. If follows that we have the  monotonicity table
\begin{center}
    \begin{tikzpicture}
  \tkzTabInit[lgt=2,espcl=4]
    {$r$ /1, $\phi_1'(r)$ /1, $\phi_1(r)$ /2}
    {$1$, $\sqrt{\frac ac}$, $\frac ac$}
  \tkzTabLine{,+,d,-,}
  \tkzTabVar{-/ ,  +D+/$+\infty$, -/ }
\end{tikzpicture}
\end{center}
Moreover $\phi_1(1)=\frac{x^2}{(a+c)^2}+\frac{y^2}{(a-c)^2}<1$. So, by the mean value principle, there exists a $r_0\in\left(1,\sqrt{\frac ac}\right)$ such that $\phi(r_0)=1$, which means that 
\[\frac{x^2}{~\displaystyle\left(\frac a{r_0}+cr_0\right)^2~}+\frac{y^2}{~\displaystyle\left(cr_0-\frac a{r_0}\right)^2~}\,=\,1.\]
Now, we can find $\theta\in\bb R$ such that 
\[
x\,=\,\left(\frac a{r_0}+cr_0\right)\cos\theta\quad\text{and}\quad 
y\,=\,\left(cr_0-\frac a{r_0}\right)\sin\theta.
\]
Hence $x+iy=F(r_0e^{i\theta})$ and $1<r_0<|a/c|$. Finally, we get
\[\mathcal E=F\left(\left\{z\in\bb C\,;\,1<|z|<\left|\frac ac\right|\right\}\right).\qedhere\]
\end{proof}
\begin{lemma}\label{Le:SkL2}
Let $d\in\bb C,\,|d|<1$ and let 
\[A\subset W=\{z\in\bb D\,;\,|d|<|z|<1.\}\]
Assume that $A$ has an accumulation point in $W$. Let $q:\bb D\to\bb C$ be an analytic function such that $q(z)=q\left(\frac dz\right)$ for all $z\in A$. Then $q$ is constant.
\end{lemma}
\begin{proof}~
    \begin{figure}[ht]
    \begin{tikzpicture}
      \draw(0,0)node{\includegraphics[scale=1]{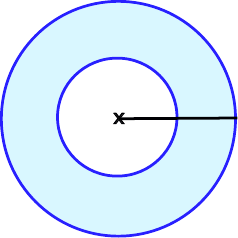}};
      \draw(1.1,-.15)node{$d$};
      \draw(2.2,0)node{$1$};
      \draw[blue](-1,-1)node{$W$};
    \end{tikzpicture}
    \caption{}
    \label{fig6}
\end{figure}

Observe that $z\longmapsto q(z)$ and $z\longmapsto q\left(\frac dz\right)$ are analytic on $W$ and coincide on $A$. By the analytic principle (since $A$ has an accumulation point in $W$), we get that for every $z\in W$, $q(z)=q\left(\frac dz\right)$. Define now
\[
\widetilde q(z)\,=\,\begin{cases}
q(z),&z\in\bb D;\\
q\left(\frac dz\right),&|z|>d.
\end{cases}
\]
Then $\widetilde q$ is a well defined analytic function on $\bb C$. Moreover, we have
\[
\lim_{|z|\to\infty}\widetilde q(z)\,=\,\lim_{|z|\to\infty} q\left(\frac dz\right)\,=\,q(0).
\]
Hence $\widetilde q$ is bounded and, by Liouville's theorem, we get that $\widetilde q$ is constant. Thus $q$ is constant as well.
\end{proof}
We can now give the Shkarin's characterization for hypercyclicity of tridiagonal Toeplitz operators.
\begin{theorem}[Shkarin, 2012 \cite{Shkarin2012}]
Let $F(e^{i\theta})=ae^{-i\theta}+b+ce^{i\theta}$, where $a,b,c\in\bb C$, $c\neq 0$. Then the following assertions are equivalent.
\begin{enumerate}
\item[(1)] $T_F$ is hypercyclic on $H^2$;
\item[(2)] both conditions are satisfied:
\begin{enumerate}[(i)]
    \item $|a|>|c|$;
    \item $\mathcal E\cap\bb T\neq\varnothing$, where $\cal E$ is the interior of the elliptic curve $F(\bb T)$.
\end{enumerate}
\end{enumerate}
\end{theorem}
\begin{proof}
$(1)\implies (2)$: let us first assume that $T_F$ is hypercyclic on $H^2$. We already seen that $(i)$ is necessary (because if $|a|\le|c|$, then $T_F$ is hyponormal and hence not hypercyclic). Let us now check $(ii)$. Argue by absurd and assume that $\mathcal E\cap \bb T=\varnothing$.
    
\begin{figure}[ht]
\begin{tikzpicture}
 \draw[->](-2,0)--(2,0);
 \draw[->](0,-2)--(0,2);
 \draw(1,0)arc(0:360:1);
x \draw(-4,.5)node{\includegraphics[scale=.7]{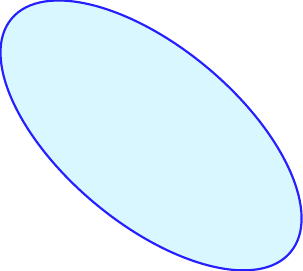}};
 \draw[blue](-4,.5)node{$\mathcal E$};
 \draw[blue](-4,-1.2)node{$F(\bb T)$};
 \draw(1.1,.1)node{$1$};
 \draw(.8,-.9)node{$\bb T$};
\end{tikzpicture}
\caption{}
\label{fig7}
\end{figure}

Hence, there exists $\theta_0\in\bb R$ such that $\Re(e^{i\theta_0}F(z))\ge1$ for every $z\in\bb T$. In particular, for every $f\in H^2$, we have:
\begin{align*}
\Re\ps{e^{i\theta_0}T_Ff,f}_2\,
&=\,\Re\left(\int_{\bb T}e^{i\theta_0}F(z)|f(z)|^2\,\mathrm dm(z)\right)\\
&=\,\int_{\bb T}\Re(e^{i\theta_0}F(z))|f(z)|^2\,\mathrm dm(z)\\
&\ge\,\int_{\bb T}|f(z)|^2\,\mathrm dm(z)\,=\,\|f\|^2_2.
\end{align*}
Hence, by Cauchy-Schwartz' inequality, we get $\|T_Ff\|_2\|f\|_2\ge\|f\|_2^2$. In particular, for every $f\in H^2,\,f\neq0$, we have $\|T_Ff\|_2\ge\|f\|_2$. In \Cref{Prop:HCcn}, we have seen that this  contradicts the hypercyclicity of $T_F$ and this gives a contradiction. Si $\mathcal E\cap\mathbb T\neq\varnothing$.

$(2)\implies (1)$: let us now assume that $(i)$ and $(ii)$ are satisfied and let us check that $T_F$ is hypercyclic on $H^2$. By \Cref{Le:SkL1}, we know that $F(\bb T)$ is an elliptic curve and its interior $\mathcal E$ is such that 
\[\mathcal E=F\left(\left\{z\in\bb C\,;\,1<|z|<\left|\frac ac\right|\right\}\right).\]

\begin{figure}[ht]
\begin{tikzpicture}
 \draw(0,0)node{\includegraphics[scale=.75]{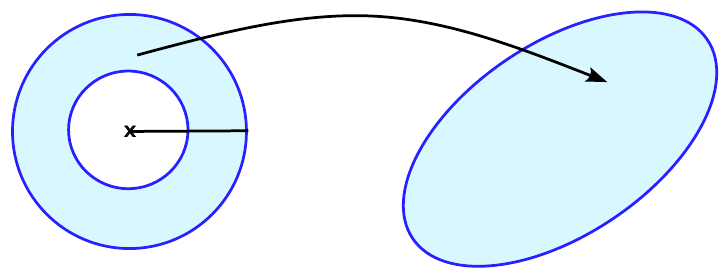}};
 \draw(0,1.8)node{$F$};
 \draw[blue](3,0)node{$\mathcal E$};
 \draw[blue](3.9,-1.15)node{$F(\bb T)$};
 \draw(-2.1,-.15)node{$1$}(-1.2,0)node{$\left|\frac ac\right|$};
 \draw(-3,-.9)node{$\bb T$};
 \draw(1.7,-1)arc(0:360:1);
 \draw(-.3,-1.65)node{$\bb T$};
\end{tikzpicture}
\caption{}
\label{fig8}
\end{figure}

Take now $\mu= F(z_0)\in\mathcal E$ for some $z_0\in\bb C$ with $1<|z_0|<\left|\frac ac\right|$. Let $f\in H^2$. Since $T_F=aS^*+bI+cS$, we have
\begin{align*}
f\in\ker(T_F-\mu I)
&\iff T_Ff=F(z_0)f\\
&\iff \forall\,z\in\bb D,~a\frac{f(z)-f(0)}z+bf(z)+czf(z)\,=\,F(z_0)f(z)\\
&\iff \forall\,z\in\bb D,~a(f(z)-f(0))+bzf(z)+cz^2f(z)=zF(z_0)f(z)\\
&\iff \forall\,z\in\bb D,~\Big(a+\big(b-F(z_0)\big)z+cz^2\Big)f(z)=af(0)\\
&\iff \forall\,z\in\bb D,~\left(a-\left(\frac a{z_0}+cz_0\right)z+cz^2\right)f(z)=af(0)\\
&\iff \forall\,z\in\bb D,~\left(z^2-\left(\frac a{cz_0}+z_0\right)z+\frac ac\right)f(z)\,=\,\frac ac f(0).
\end{align*}
Now observe that 
\[z^2-\left(\frac a{cz_0}+z_0\right)z+\frac ac\,=\,(z-z_0)\left(z-\frac a{cz_0}\right)\]
and since $1<|z_0|<\left|\frac ac\right|$, the function 
\[z\longmapsto\frac 1{~\displaystyle (z-z_0)\left(z-\frac a{cz_0}\right)~}\]
belongs to $H^2$. Thus
\[\ker(T_F-\mu I)\,=\,\bb C\cdot \frac 1{~\displaystyle (z-z_0)\left(z-\frac a{cz_0}\right)~}\]
where $\mu= F(z_0)$ and $1<|z_0|<\left|\frac ac\right|$. Since $\mathcal E\cap\bb T\neq\varnothing$, the two open sets $\mathcal O_1=\mathcal E\cap\bb D$ and $\mathcal O_2=\mathcal E\cap(\bb C\setminus\overline{\bb D})$ are non empty. Denote by $\Omega_1=F^{-1}(\mathcal O_1)$ and $\Omega_2=F^{-1}(\mathcal O_2)$. Then $\Omega_1$ and $\Omega_2$ are two non empty open subsets of $\left\{1<|z_0|<\left|\frac ac\right|\right\}$ and
\[\spa\left(\frac 1{~\displaystyle (z-z_0)\left(z-\frac a{cz_0}\right)~}\,;\,z_0\in\Omega_1\right)\subset H_-(T_F),\]
\[\spa\left(\frac 1{~\displaystyle (z-z_0)\left(z-\frac a{cz_0}\right)~}\,;\,z_0\in\Omega_2\right)\subset H_+(T_F).\]
By \ref{GScrit}, it is sufficient to check that if $\Omega$ is a non empty open set of $\left\{1<|z_0|<\left|\frac ac\right|\right\}$ then 
\[\spa\left(\frac 1{~\displaystyle (z-z_0)\left(z-\frac a{cz_0}\right)~}\,;\,z_0\in\Omega\right)~\text{is dense in}~H^2.\]
So let $h\in H^2$ such that for every $z_0\in\Omega$, we have 
\[h\perp \frac 1{~\displaystyle (z-z_0)\left(z-\frac a{cz_0}\right)~}.\]
Note that if $z_0^2\neq a/c$, we have
\begin{align*}
\frac 1{~\displaystyle (z-z_0)\left(z-\frac a{cz_0}\right)~}\,
&=\,\frac1{~\displaystyle z_0-\frac a{cz_0}~}\left(\frac1{z-z_0}-\frac1{~\displaystyle z-\frac a{cz_0}~}\right)\\
&=\,\frac1{~\displaystyle z_0-\frac a{cz_0}~}\left(\frac{-1/z_0}{~\displaystyle 1-\frac z{z_0}~}+\frac{cz_0/a}{~\displaystyle 1-\frac {cz_0}az~}\right)\\
&=\,\frac1{~\displaystyle z_0-\frac a{cz_0}~}\left(-\frac1{z_0}k_{\frac1{\overline{z_0}}}+\frac{cz_0}ak_{\frac{\overline{cz_0}}{\overline{a}}}\right),
\end{align*}
where we recall that $k_w$ denotes the reproducing kernel at $w$ of $H^2$ (observe that $\frac1{\overline{z_0}}\in\bb D$ and $\frac{\overline{cz_0}}{\overline{a}}\in \bb D$). Hence for every $z_0\in\Omega$ with $z_0^2\neq a/c$, we have
\begin{align*}
0\,
&=\,\ps{h,-\frac1{z_0}k_{\frac1{\overline{z_0}}}+\frac{cz_0}ak_{\frac{\overline{cz_0}}{\overline{a}}}}\\
&=\, -\frac1{\overline{z_0}}h\left(\frac1{\overline{z_0}}\right)+\frac{\overline{cz_0}}{\overline{a}}h\left(\frac{\overline{cz_0}}{\overline{a}}\right).
\end{align*}
Thus for every $z_0\in\Omega$ with $z_0^2\neq a/c$, we have
\[\frac1{z_0}\overline{h\left(\frac1{\overline{z_0}}\right)}\,=\,\frac{cz_0}a\overline{h\left(\frac{\overline{cz_0}}{\overline{a}}\right)}.\]
Denote by $q(z)=z\overline{h(\bar z)}$ for $z\in\bb D$. We know that $q$ is analytic on $\bb D$ and if $A=\{z_0^{-1}\,;\,z_0\in\Omega,z_0^2\neq z/c\}$, then $A$ is a nonempty open subset of $W=\{z\in\bb C\,;\,|c/a|<|z|<1\}$ and for all $z\in A$, $q(z)=q\left(\frac a{cz}\right)$. According to \Cref{Le:SkL2}, we conclude that  $q$ is constant. But $q(0)=0$, whence $q\equiv0$ and therefore $h\equiv0$.
\end{proof}
\subsection{Some more recent results}
To go a little bit further, we conclude by mentioning three recent papers.

The first one is by Baranov - Lishanskii, 2016 \cite{BaranovLishanskii2016}. They investigate the more general case where $F$ has the form 
\[F(e^{i\theta})=P(e^{-i\theta})+\phi(e^{i\theta}),\]
where $P$ is an analytic polynomial and $\phi\in A(\bb D)$. They provided some necessary conditions (of a spectral kind) for $T_F$ to be hypercyclic on $H^2$ as well as some sufficient conditions based on an explicit description of the eigenvectors of $T_F$ and on the Godefroy - Shapiro criterion. A novel feature of their conditions is the role of univalence or $N$-valence (where $N$ is the degree of $P$) of the symbol. The same line of approach was taken in the subsequent work \cite{AbakumovBaranovCharpentierLishanskii2021} of Abakumov, Baranov, Charpentier and Lishanskii where they extended results of \cite{BaranovLishanskii2016} to the case of more general symbol $F$ of the form 
\[F(e^{i\theta})=R(e^{-i\theta})+\phi(e^{i\theta}),\]
where $R$ is a rational function without poles in $\overline{\bb D}$ and $\phi\in A(\bb D)$. The novel feature of the approach taken in \cite{AbakumovBaranovCharpentierLishanskii2021} is the use of deep results of Solomyak providing necessary and sufficient conditions for finite sets of functions to be dense in $H^2$.

Finally, using the model theory for Toeplitz operators with smooth symbols developed by D. Yakubovich in the 80's, Fricain - Grivaux - Ostermann obtained in \cite{FricainGrivauxOstermann_preprint} some new necessary and sufficient conditions for Toeplitz operators to be hypercyclic on $H^p$, $1<p<\infty$. We will take some time to discuss the flavor of some results obtained in \cite{FricainGrivauxOstermann_preprint}. 
Let  $q$ be the conjugate exponent of $p$, that is $\frac1p+\frac1q=1$. Consider the following conditions on the symbol $F$:
\begin{enumerate}[(H1)]
    \item\label{H1}$F$ belongs to the class $C^{1+\varepsilon}(\mathbb T)$ for some $\varepsilon>\max(1/p,1/q)$, and its derivative $F'$ does not vanish on $\mathbb T$;
   \item\label{H2} the curve $F(\mathbb T)$ self-intersects a finite number of times, i.e. 
   the unit circle $\mathbb T$ can be partitioned into a finite number of closed arcs $\alpha_1,\dots,\alpha_m$ such that
     \begin{enumerate}[(a)]
       \item $F$ is injective on the interior of each arc $\alpha_j,~1\le j \le m$;
        \item for every $i\neq j,~1\le i,j\le m$, the sets $F(\alpha_j)$ and $F(\alpha_j)$ have disjoint interiors;
    \end{enumerate}
    \item\label{H3}for every $\lambda\in\mathbb C\setminus F(\mathbb T)$, $\w_F(\lambda)\le0$, where we recall that $\w_F(\lambda)$ denotes the winding number of the curve $F(\mathbb T)$ around $\lambda$.
 \end{enumerate}
 Recall that $\sigma(T_F)=F(\bb T)\cup\{\lambda\in\bb C\setminus F(\bb T)\,;\,\w_F(\lambda)<0\}$, and according to \eqref{eq:windnumber-spectre}, the condition \ref{H3} is a necessary condition for hypercyclicity. We denote by $\mathcal C$ the set of all connected component of $\sigma(T_F)\setminus F(\bb T)$. As we have seen in \Cref{Th:HcIntSp}, if $T$ is a hypercyclic operator on a separable Banach space $X$, every component of the spectrum of $T$ must interset $\bb T$. It turns out that, in our current setting, a stronger property must hold.
 \begin{theorem}[Fricain - Grivaux - Ostermann, 2025 \cite{FricainGrivauxOstermann_preprint}]\label{Thm-CN}
 Let $p>1$ and let $F$ satisfy \ref{H1}, \ref{H2} and \ref{H3}. If $T_F$ is hypercyclic on $H^p$, then every component of the interior of the spectrum of $T_F$ must intersect $\bb T$.
 \end{theorem}

The proof of Theorem \ref{Thm-CN} relies on the fact that under conditions \ref{H1}, \ref{H2} and \ref{H3}, the operator $T_F$ has an $H^\infty$-functional calculus on the interior of its spectrum. 
 
 In the other direction, we have the following result.
  \begin{theorem}[Fricain - Grivaux - Ostermann, 2025 \cite{FricainGrivauxOstermann_preprint}]
 Let $p>1$ and let $F$ satisfy \ref{H1}, \ref{H2} and \ref{H3}. Suppose that, for every $\Omega\in\mathcal C$, $\Omega\cap\bb T\neq\varnothing$. Then $T_F$ is hypercyclic on $H^p$.
 \end{theorem}
 When the set $\sigma(T_F)\setminus F(\bb T)$ is connected, we then deduce the following complete characterization.
  \begin{corollary}[Fricain - Grivaux - Ostermann, 2025 \cite{FricainGrivauxOstermann_preprint}]
 Let $p>1$ and let $F$ satisfy \ref{H1}, \ref{H2} and \ref{H3}. Suppose that $\sigma(T_F)\setminus F(\bb T)$ is connected. Then, the following are equivalent:
 \begin{enumerate}[(i)]
     \item $T_F$ is hypercyclic on $H^p$;
     \item $\bb T\cap \overset{\circ}{\sigma(T_F)}\neq\varnothing$.
 \end{enumerate}
 \end{corollary}
We refer to \cite{FricainGrivauxOstermann_preprint} for more deep results and proof of all the previous results. 

\bibliographystyle{plain}
\bibliography{biblio}

\begin{bibdiv}
\begin{biblist}

\bib{AbakumovBaranovCharpentierLishanskii2021}{article}{
      author={Abakumov, E.},
      author={Baranov, A.},
      author={Charpentier, S.},
      author={Lishanskii, A.},
       title={New classes of hypercyclic {T}oeplitz operators},
        date={2021},
        ISSN={0007-4497},
     journal={Bull. Sci. Math.},
      volume={168},
       pages={Paper No. 102971, 13},
         url={https://doi.org/10.1016/j.bulsci.2021.102971},
      review={\MR{4237429}},
}

\bib{BaranovLishanskii2016}{article}{
      author={Baranov, A.},
      author={Lishanskii, A.},
       title={Hypercyclic {T}oeplitz operators},
        date={2016},
        ISSN={1422-6383},
     journal={Results Math.},
      volume={70},
      number={3-4},
       pages={337\ndash 347},
         url={https://doi.org/10.1007/s00025-016-0527-x},
      review={\MR{3544864}},
}

\bib{BayartMatheron2009}{book}{
      author={Bayart, F.},
      author={Matheron, \'{E}.},
       title={Dynamics of linear operators},
      series={Cambridge Tracts in Mathematics},
   publisher={Cambridge University Press, Cambridge},
        date={2009},
      volume={179},
        ISBN={978-0-521-51496-5},
         url={https://doi-org.acces.bibl.ulaval.ca/10.1017/CBO9780511581113},
      review={\MR{2533318}},
}

\bib{Birkhoff1922}{article}{
      author={Birkhoff, George~D.},
       title={Surface transformations and their dynamical applications},
        date={1922},
        ISSN={0001-5962},
     journal={Acta Math.},
      volume={43},
      number={1},
       pages={1\ndash 119},
         url={https://doi-org.acces.bibl.ulaval.ca/10.1007/BF02401754},
      review={\MR{1555175}},
}

\bib{Birkhoff1929}{article}{
      author={Birkhoff, George~David},
       title={D{\'e}monstration d'un th{\'e}or{\`e}me {\'e}l{\'e}mentaire sur
  les fonctions enti{\`e}res},
        date={1929},
     journal={CR Acad Sci. Paris Ser. I Math.},
      volume={189},
       pages={473\ndash 475},
}

\bib{BrownHalmos1963}{article}{
      author={Brown, Arlen},
      author={Halmos, P.~R.},
       title={Algebraic properties of {T}oeplitz operators},
        date={1963/64},
        ISSN={0075-4102},
     journal={J. Reine Angew. Math.},
      volume={213},
       pages={89\ndash 102},
  url={https://doi-org.acces.bibl.ulaval.ca/10.1007/978-1-4613-8208-9_19},
      review={\MR{160136}},
}

\bib{Coburn1966}{article}{
      author={Coburn, L.~A.},
       title={Weyl's theorem for nonnormal operators},
        date={1966},
        ISSN={0026-2285},
     journal={Michigan Math. J.},
      volume={13},
       pages={285\ndash 288},
  url={http://projecteuclid.org.acces.bibl.ulaval.ca/euclid.mmj/1031732778},
      review={\MR{201969}},
}

\bib{FricainGrivauxOstermann_preprint}{misc}{
      author={Fricain, Emmanuel},
      author={Grivaux, Sophie},
      author={Ostermann, Maëva},
       title={Hypercyclicity of {T}oeplitz operators with smooth symbols},
        date={arXiv:2502.03303},
}

\bib{GodefroyShapiro1991}{article}{
      author={Godefroy, G.},
      author={Shapiro, J.~H.},
       title={Operators with dense, invariant, cyclic vector manifolds},
        date={1991},
        ISSN={0022-1236},
     journal={J. Funct. Anal.},
      volume={98},
      number={2},
       pages={229\ndash 269},
         url={https://doi.org/10.1016/0022-1236(91)90078-J},
      review={\MR{1111569}},
}

\bib{GrosseErdmannPeris2011}{book}{
      author={Grosse-Erdmann, K.-G.},
      author={Peris~Manguillot, A.},
       title={Linear chaos},
      series={Universitext},
   publisher={Springer, London},
        date={2011},
        ISBN={978-1-4471-2169-5},
         url={https://doi.org/10.1007/978-1-4471-2170-1},
      review={\MR{2919812}},
}

\bib{HartmanWintner1954}{article}{
      author={Hartman, Philip},
      author={Wintner, Aurel},
       title={The spectra of {T}oeplitz's matrices},
        date={1954},
        ISSN={0002-9327},
     journal={Amer. J. Math.},
      volume={76},
       pages={867\ndash 882},
         url={https://doi-org.acces.bibl.ulaval.ca/10.2307/2372661},
      review={\MR{73859}},
}

\bib{Kitai1982}{book}{
      author={Kitai, Carol},
       title={Invariant closed sets for linear operators},
   publisher={ProQuest LLC, Ann Arbor, MI},
        date={1982},
        ISBN={978-0315-10381-8},
  url={http://gateway.proquest.com.acces.bibl.ulaval.ca/openurl?url_ver=Z39.88-2004&rft_val_fmt=info:ofi/fmt:kev:mtx:dissertation&res_dat=xri:pqdiss&rft_dat=xri:pqdiss:NK58297},
        note={Thesis (Ph.D.)--University of Toronto (Canada)},
      review={\MR{2632793}},
}

\bib{MacLane1952}{article}{
      author={MacLane, G.~R.},
       title={Sequences of derivatives and normal families},
        date={1952},
        ISSN={0021-7670},
     journal={J. Analyse Math.},
      volume={2},
       pages={72\ndash 87},
         url={https://doi-org.acces.bibl.ulaval.ca/10.1007/BF02786968},
      review={\MR{53231}},
}

\bib{MartinezAvendanoRosenthalB}{book}{
      author={Mart\'{\i}nez-Avenda\~{n}o, Rub\'{e}n~A.},
      author={Rosenthal, Peter},
       title={An introduction to operators on the {H}ardy-{H}ilbert space},
      series={Graduate Texts in Mathematics},
   publisher={Springer, New York},
        date={2007},
      volume={237},
        ISBN={978-0-387-35418-7; 0-387-35418-2},
      review={\MR{2270722}},
}

\bib{Rolewicz1969}{article}{
      author={Rolewicz, S.},
       title={On orbits of elements},
        date={1969},
        ISSN={0039-3223},
     journal={Studia Math.},
      volume={32},
       pages={17\ndash 22},
         url={https://doi-org.acces.bibl.ulaval.ca/10.4064/sm-32-1-17-22},
      review={\MR{241956}},
}

\bib{Shkarin2012}{misc}{
      author={Shkarin, S.},
       title={Orbits of coanalytic {T}oeplitz operators and weak
  hypercyclicity},
        date={arXiv:1210.3191},
         url={https://arxiv.org/abs/1210.3191},
}

\end{biblist}
\end{bibdiv}
\end{document}